\documentclass[12pt]{article}

\usepackage[colorlinks=true,citecolor=black,linkcolor=black,urlcolor=blue]{hyperref}
\usepackage[italian, english]{babel}
\usepackage[utf8]{inputenc}
\usepackage{latexsym}
\usepackage{booktabs} 
\usepackage{amsfonts}
\usepackage[pdftex]{graphicx}
\usepackage{subfig}
\usepackage{caption}
\usepackage{textcomp} 
\usepackage{amsmath,amssymb,amsthm}
\usepackage{tikz,color,graphicx}
\usepackage{pgfplots}
\usepackage{floatflt,epsfig}
\usepackage{bbm}
\usepackage{blkarray}
\usepackage{kbordermatrix}

\usepackage{tikz}
\usetikzlibrary{positioning,chains,fit,shapes,calc}
\usetikzlibrary{arrows,intersections}

\theoremstyle{plain}
\newtheorem{teorema}{Theorem}[section]
\theoremstyle{definition}
\newtheorem{definizione}[teorema]{Definition}
\theoremstyle{definition}
\newtheorem{ass}[teorema]{Assumption}
\theoremstyle{definition}

\theoremstyle{definition}

\theoremstyle{definition}

\theoremstyle{plain}

\theoremstyle{plain}
\newtheorem{proposizione}[teorema]{Proposition}
\theoremstyle{plain}
\newtheorem{lemma}[teorema]{Lemma}
\theoremstyle{plain}

\theoremstyle{definition}
\newtheorem{remark}[teorema]{Remark}
\theoremstyle{plain}

\numberwithin{equation}{section}
\numberwithin{teorema}{section}

\def\dfrac#1#2{\lower0.15ex\hbox{\large$\frac{#1}{#2}$}}

\begin{document}

\pagestyle{plain}
\pagenumbering{arabic}


\title{Adding edge dynamics to bipartite random-access networks}
\author{\renewcommand{\thefootnote}{\arabic{footnote}}
\renewcommand{\thefootnote}{\arabic{footnote}}
Matteo Sfragara\,%
\footnotemark[1]}
\footnotetext[1]{%
Mathematical Institute, Leiden University, The Netherlands}
\date{\today}
\maketitle

\begin{abstract}

We consider random-access networks with nodes representing transmitter-receiver pairs whose signals interfere with each other depending on their vicinity. Data packets arrive at the nodes over time and form queues. The nodes can be either active or inactive: a node deactivates at unit rate, while it activates at a rate that depends on its queue length, provided none of its neighbors is active. In order to model the effects of user mobility in wireless networks, we analyze dynamic interference graphs where the edges are allowed to appear and disappear over time. We focus on bipartite graphs and study the transition time between the two states where one part of the network is active and the other part is inactive, in the limit as the queue lengths become large. Depending on the speed of the dynamics, we are able to obtain a rough classification of the effects of the dynamics on the transition time.

\vspace{1cm}
\end{abstract}

\newpage

\section{Motivation and background}

The present paper is a continuation of \cite{BdHNS18} and \cite{BdHNS20}. Our main goal is to extend the metastability analysis of random-access models for static networks to dynamic networks where edges appear and disappear over time. In Section~\ref{section1.1} we introduce random-access protocols for wireless networks and discuss the importance of studying their metastability properties (we refer to \cite[Section 1.1]{BdHNS18} and \cite[Section 1.1]{BdHNS20} for further background). In Section~\ref{section1.2} we motivate our interest in adding edge dynamics and letting the interference graph change over time. This represents a natural basic model to capture the effects of user mobility in wireless networks. In Section~\ref{section1.3} we turn our focus on bipartite networks and we explain how we build on the theory from \cite{BdHNS18} and \cite{BdHNS20}.

\subsection{Wireless random-access networks}
\label{section1.1}
Wireless communication consists of the transmission of data or information, without any conductor, from one device (transmitter) to another (receiver) through radio frequency and radio signals. In order to improve their performance and reduce collisions, which occur if nearby ongoing conflicting transmissions interfere with each other, wireless networks require a medium access control mechanism. Many such mechanisms have been proposed and analyzed and they can be mainly divided into two classes. In centralized algorithms, a global control entity has perfect information of all the interference constraints and coordinates all the transmissions by prescribing certain scheduling to the devices in the network. In distributed algorithms, the devices decide autonomously when to start a transmission using only local information. Most distributed algorithms involve randomness to avoid simultaneous transmissions and share the medium in the most efficient way. Thanks to their low implementation complexity, randomized algorithms, also called random-access algorithms, have become a popular mechanism for distributed-medium access control. 

The main idea behind random-access algorithms is to associate with each device a random clock, independently of all the other devices. The clock determines when the device attempts to access the medium in order to transmit. This back-off mechanism was developed to avoid simultaneous activity of nearby devices and to reduce the chances of collisions. Even though random-access algorithms can be described in a simple way and only require local information, their macroscopic behavior in large networks tends to be very complex: the network performance critically depends on the global spatial characteristics and the geometry of the network (see \cite{BB09t}, \cite{BB09a}). Indeed, nearby devices are typically prevented from simultaneous transmission in order to prevent them to interfere and to disturb each other’s signals.

The Carrier-Sense Multiple-Access (CSMA) algorithm is a collision avoidance protocol that combines the random back-off mechanism with interference sensing (see \cite{KT75}). The devices first sense the shared medium and only start a packet transmission if no ongoing transmission activity from interfering devices is detected. They attempt to transmit after a random back-off time, but if they sense activity of interfering devices, then they freeze their back-off timer until the transmission medium is sensed idle again. Note that several devices can transmit by accessing the same medium alternately. CSMA algorithms try to ensure that devices do not start a transmission at the same time in order to prevent collisions. They are popular in distributed random-access networks and various versions are currently implemented in IEEE 802.11 Wi-Fi networks.

Random-access networks with CSMA protocols can be modeled as interacting particle systems on graphs with hard-core interaction (see \cite{Z15}). The undirected graph, which we refer to as the interference graph, describes the conflicting transmissions of the devices due to interference. Each transmitter-receiver pair is represented by a particle, which is active when data packets are being transmitted and inactive otherwise. The interference graph encodes the spatial characteristics and the structure of the network, since neighboring particles are not allowed to be active simultaneously. 

\begin{figure}[htbp]
\begin{center}
\vspace{0.5cm}
\includegraphics[width=.45\linewidth]{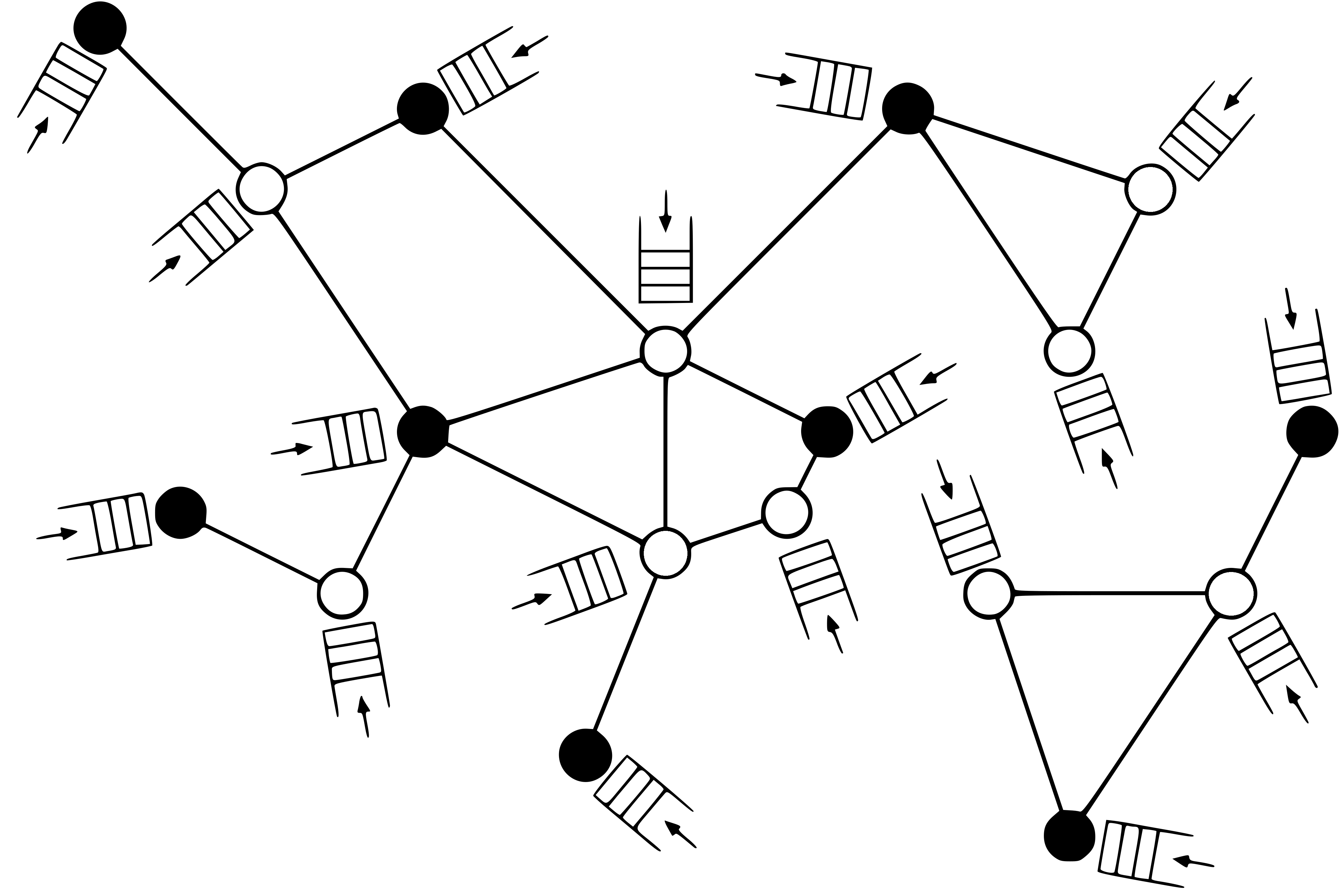}
\vspace{0.5cm}
\caption{\small A random-access network, where each node represents a transmitter-receiver pair with a queue of data packets.}
\label{fig:network}
\end{center}
\vspace{-0.5cm}
\end{figure}

Metastability is the phenomenon where a system, when subjected to a small noise, moves between different regions of its state space on different time scales (see \cite{BdH05}). A metastable state is a quasi-equilibrium that persists on a short time scale and represents a configuration where the energy of the system has a local minimum. On a long time scale, the system reaches an equilibrium, called a stable state, which represents a configuration where the energy of the system has a global minimum. The transition time between the two states depends on the depth of the energy valley around the metastable state and the shape of the bottleneck separating them. It is known that hard-core interaction models exhibit metastability effects, in the sense that the transition from a metastable to a stable state takes a very long time. Indeed, as the activation rates become large, the network tends to stabilize in configurations with the maximum number of active nodes and the transitions between these configurations become extremely slow. As a consequence, over finite time intervals, some nodes are excluded from activity, while other nodes are active essentially all the time. Although the aggregate throughput may improve, it has been observed that some nodes experience prolonged starvation, which results in a significant increase of their queues and long delays (see \cite{BCvL14}, \cite{DDT18}, \cite{DDT09}, \cite{GSK08}, \cite{GBW14}, \cite{SA11}, \cite{Z15}). Understanding metastability properties is of great practical importance, since they represent a major tool to analyze how likely such unfairness and starvation issues persist over time. Moreover, they are crucial to designing mechanisms that counter these effects and improve the overall performance of the network.

\subsection{User mobility}
\label{section1.2}

User \textit{mobility} is one of the major features in wireless networks. Different mobility patterns can be distinguished (pedestrians, vehicles, aerial, dynamic medium, robot, and outer space motion) and mathematical models can be developed in order to generalize such patterns and analyze their characteristics. Understanding the effects of user mobility in wireless networks is crucial in order to design efficient protocols and improve network performance.

Consider device-to-device (D2D) communication in cellular networks, defined as direct communication between devices without involving a base station. In recent years, there has been a renewed interest in D2D communication, due to a rapidly increasing demand for cellular data communication and the introduction of new proximity-based services. The D2D communication can provide a significant capacity gain by enabling the network to offload data traffic to direct communication links between devices. However, without a proper interference control mechanism, a D2D link can generate serious interference to other D2D links. To protect a D2D receiver, fully-distributed random access protocols, which employ a competition mechanism conceptually similar to CSMA protocols, have been implemented by creating an exclusion region around the receiver, where interfering devices are prohibited from transmitting (see \cite{ZCK15}). Interest has also been drawn towards sensor networks, where significant challenges arise when dealing with the coordination of the increasing number of devices accessing shared resources. Great effort has been invested in reducing the random-access collisions in order to efficiently manage resources in the context of D2D interaction (see \cite{AHK17}, \cite{HHN18}). Encouraged by the popularity of using D2D communications and unlicensed bands in cellular systems (see \cite{LCTH16}, \cite{WGYLWCZ16}), various new congestion control approaches based on D2D communications and device grouping have been developed due to their advantages in energy consumption and access delay (see \cite{HSHDS19}, \cite{SGBHP16}). In this setting of networks with D2D communication, when allowing user mobility, new groups of devices are formed and, since their signals interfere with each other depending on their vicinity, the interference graph of the network changes over time. 

To the best of our knowledge, random-access models with user mobility in the context of interference graphs have so far not been considered in the literature. All the studies we are aware of that have examined the impact of user mobility in wireless networks are concerned with handover mechanisms (see \cite{R02}, \cite{TV05}), so-called opportunistic scheduling algorithms (see \cite{BHP09}, \cite{VTL02}), capacity issues in ad hoc and cellular networks (see \cite{BPH06}, \cite{GT02}), and flow-level performance (see \cite{BBP04}, \cite{BBHJP09}, \cite{BBP04}, \cite{BS13}, \cite{SS20}).

In this paper we investigate a dynamic version of the random-access protocols in order to try to capture some features of user mobility in wireless networks. A natural paradigm for constructing \textit{dynamic interference graphs} would be to use geometric graphs, such as unit-disk graphs. Each node, representing a transmitter-receiver pair, follows a random trajectory and experiences interference from all nodes within a certain distance. A feasible state of the interference graph would then be generated by a specific instance of the geometric graph. We follow a different approach and, with an explorative intention, we consider a model where edges are allowed to appear and disappear from the graph according to i.i.d.\ Poisson clocks placed on each of them. 

\subsection{Bipartite interference graphs}
\label{section1.3}

The work in \cite{BdHNS18} focuses on random-access networks modeled by complete bipartite interference graphs, where the nodes are partitioned into two nonempty sets $U$ and $V$. In the limit as the queue lengths become large, starting from the configuration in which all the nodes in $U$ are active and all the nodes in $V$ are inactive, it is of particular interest to determine the transition time to the configuration where all the nodes in $U$ are inactive and all the nodes in $V$ are active, since it represents a metastable crossover. In \cite{BdHNS20}, the model is extended to arbitrary bipartite interference graphs, where not necessarily all the nodes in $U$ interfere with all the nodes in $V$. This turns out to be considerably more challenging, since the transition depends on the full structure of the graph. The problem is solved thanks to a greedy algorithm that identifies sequences of subtransitions together with their probabilities and allows to determine the mean transition time and its distribution on the scale of its mean. It is assumed that the activation rates of the nodes in $V$ are more aggressive than in $U$ to ensure that the transition can be decomposed into a sequence of subtransitions. Both the models analyzed in \cite{BdHNS18} and \cite{BdHNS20} deal with random-access protocols for static networks. Their novelty consists in assuming queue-based activation rates, which raise challenging problems and have recently attracted great attention (see the discussion in Section~\ref{sectiondiscussion}).

In this paper we extend these queue-based random access protocols to networks modeled by dynamic bipartite graphs, where we add edge dynamics in order to capture some effects of user mobility. Although there is no specific physical reason for keeping the focus on bipartite graphs, this assumption allows mathematical tractability and represents an important step toward more general network topologies. For dynamic bipartite networks, we build on the theory of \cite{BdHNS18} and \cite{BdHNS20}, and in particular of the above-mentioned algorithm, and we combine it with a detailed analysis of the effects of the dynamics on the degrees of the nodes. Our approach is based on the intuition that a node in $V$ can activate either when its neighbors in $U$ are simultaneously inactive (as for static networks) or when the edges connecting it with its neighbors in $U$ disappear. Interpolation between these two situations gives rise to different scenarios and interesting behavior. In the limit as the queue lengths become large, we are able to describe the typical behavior of the network and identify how the mean transition time depends on the speed of the dynamics.

\section{Model description}
\label{modeldescription}

In this section we define the mathematical model of our interest, in line with the models in \cite{BdHNS18} and \cite{BdHNS20}. We first describe how static networks can be modeled by arbitrary bipartite interference graphs. We then add edge dynamics and discuss how the graph evolves over time.

\subsection{The static model} 
\label{ss:model}

We consider the bipartite graph $G=(U \sqcup V, E)$, where $U \sqcup V$ is the set of nodes and $E$ is the set of edges that connect a node in $U$ to a node in $V$, and vice versa (edges are undirected).

A node can be either \textit{active} or \textit{inactive}. The state of node $w$ at time $t$ is described by a Bernoulli random variable $X_w(t) \in \{0,1\}$, defined as
\begin{equation} 
X_w(t) = 
\begin{cases} 
0, \text{ if } w \text{ is inactive at time } t ,\\ 
1, \text{ if } w \text{ is active at time } t. 
\end{cases}
\end{equation} 
The \textit{joint node activity state} at time $t$ is denoted by
\begin{equation}
{\bf X}(t) = \{ X_w(t) \}_{w \in U \sqcup V}
\end{equation}
and it is an element of the state space
\begin{equation}
\mathcal{X} = \big\{{\bf X} \in \{0,1\}^{|U \sqcup V|} \colon\, X_iX_j = 0\,\,\,\forall\, (i,j) \in E\big\},
\end{equation}
where $X_i = 0$ means that node $i$ is inactive and $X_i =1$ that it is active. We denote by ${\bf 1_U}$ (${\bf 1_V}$) the configuration where all nodes in $U$ are active (inactive) and all nodes in $V$ are inactive (active).

An active node $w$ becomes inactive according to a \textit{deactivation Poisson clock}: when the clock ticks the node deactivates. Conversely, an inactive node $w$ attempts to become active according to an \textit{activation Poisson clock}, but the attempt is successful only when no neighbors of $w$ are active. While the deactivation rate is $1$, interesting scenarios arise when the activation rate at node $w$ at time $t$ depends on the queue length at node $w$ at time $t$.

\begin{definizione}[\bf{Queue length at a node}]
Let $t \mapsto Q_w^+(t)$ be the \textit{input process} describing packets arriving at node $w$ according to a Poisson process $t \mapsto N_w(t) = \mathrm{Poisson}(\sigma t)$ and requiring i.i.d.\ exponential service times $Y_{wn}$, $n \in \mathbb{N}$, with rate $\xi_U$ for $w \in U$ and $\xi_V$ for $w \in V$. This is a compound Poisson process with mean $\rho_U = \sigma / \xi_U$ for $w \in U$ and $\rho_V = \sigma / \xi_V$ for $w \in V$. Let $t \mapsto Q_w^-(t)$ be the \textit{output process} representing the cumulative amount of work that is processed by the server at node $w$ in the time interval $[0,t]$ at speed $c$, which equals $cT_w(t) = c \int_0^t X_w(s) ds$. In order to ensure that the queue length tends to decrease when a node is active, we assume that $\rho_U < c$ and $\rho_V < c$. Define
\begin{equation} 
\Delta_w(t) = Q_w^+(t) - Q_w^-(t) = \sum_{n=0}^{N_w(t)} Y_{wn} - c T_w(t)
\end{equation}
and let $s^* = s^*(t)$ be the value where $\sup_{s \in [0,t]} [\Delta_w(t) - \Delta_w(s)]$ is reached, i.e., equals $[\Delta_w(t) - \Delta_w(s^*-)]$. Let $Q_w(t) \in \mathbb{R}_{\geq 0}$ denote the \textit{queue length} at node $w$ at time $t$. Then
\begin{equation}
Q_w(t) = \max\big\{ Q_w(0) + \Delta_w(t),\,\Delta_w(t)-\Delta_w(s^*-) \big\},
\end{equation}
where $Q_w(0)$ is the initial queue length. The maximum is achieved by the first term when $Q_w(0) \geq -\Delta_w(s^*-)$ (the queue length never sojourns at $0$), and by the second term when $Q_w(0) < -\Delta_w(s^*-)$ (the queue length sojourns at $0$ at time $s^*-$).
\end{definizione}

The \textit{initial queue length} is assumed to be given by
\begin{equation}
\label{initialqueues}
Q_w(0) = 
\left\{\begin{array}{ll} 
\gamma_U r, &w \in U, \\  
\gamma_V r, &w \in V, 
\end{array}
\right.
\end{equation}
where $\gamma_U \geq \gamma_V > 0$, and $r$ is a parameter that tends to infinity.

We now focus on a queue-based model where the activation rates depend on the current state of the queues. Later, in Section~\ref{modelwithfixedrates} we will also consider a simplified version of the model with fixed activation rates.

\begin{ass}[{\bf Queue-based activation rates}]
\label{queuedependentrates}
Let $g_U, g_V \in \mathcal{G}$ with
\begin{equation}
\mathcal{G} = \Big\{g\colon\,\mathbb{R}_{\geq 0} \to \mathbb{R}_{\geq 0}\colon\,
g \text{ non-decreasing and continuous},\, g(0)=0,\, \lim_{x \to \infty} g(x) = \infty\Big\}.
\end{equation}
The deactivation clocks tick at rate $1$, while the activation clocks tick at rate 
\begin{equation}
\phi_w(t) = 
\left\{\begin{array}{ll} 
g_U(Q_w(t)), &w \in U, \\ 
g_V(Q_w(t)), &w \in V,
\end{array}
\right. 
\qquad t \geq  0.
\end{equation}
We focus on the particular choice
\begin{equation} 
\begin{array}{ll}
g_U(x) = B x^{\beta}, &x \in [0, \infty), \\
g_V(x) = B' x^{\beta'}, &x \in [0, \infty),
\end{array}
\end{equation}
with $B,B', \beta, \beta' \in (0, \infty)$. We assume that nodes in $V$ are {\it much more aggressive} than nodes in $U$, namely,
\begin{equation}
\beta' > \beta +1.
\end{equation}
This ensures that the transition from ${\bf 1_U}$ to ${\bf 1_V}$ can be decomposed into a succession of transitions on complete bipartite subgraphs. Note that, from a metastability perspective, ${\bf 1_U}$ represents the metastable state, while ${\bf 1_V}$ represents the stable state.
\end{ass}

We next introduce the transition time of the graph $G$, which is our main object of interest in this paper.
\begin{definizione}[{\bf Transition time}]
Let ${\bf Q_U} = \{ Q_{U,i} \}_{i = 1}^{|U|}$ be the sequence of queues associated with the nodes in $U$, and ${\bf Q_V} = \{ Q_{V,j} \}_{j = 1}^{|V|}$ the sequence of queues associated with the nodes in $V$. We denote by $\mathcal{T}_G^Q$ the \textit{transition time} of the graph $G$ conditional on the initial queue lengths ${\bf Q}=({\bf Q_U},{\bf Q_V})$ and we define it as
\begin{equation} \label{transition}
\mathcal{T}_G^Q= \min \big\{t \geq 0\colon\, {\bf X}(t) = {\bf 1_V}\} \quad \text{ given } \quad {\bf X}(0) = {\bf 1_U}.
\end{equation}
It represents the time it takes to hit configuration ${\bf 1_V}$ starting from configuration ${\bf 1_U}$.
\end{definizione}

Note that the transition time is a random variable that depends on the initial queue lengths, hence all the expectations should be interpreted as conditional expectations.

\subsection{Adding edge dynamics}
We are interested in analyzing the behavior of the network when we allow the interference graph to change over time. We define a dynamic version of the model, which we are going to study in this paper.

\begin{definizione}[\bf{Dynamic interefence graphs}]
\label{wn3:def:dynamic}
We say that the interference graph is \textit{dynamic} when the edges appear and disappear according to a continuous-time flip process. Consider the dynamic bipartite interference graph $G(\cdot) = (U \sqcup V, E(\cdot))$, where $U \sqcup V$ is the set of nodes, with $|U|=M$ and $|V|=N$, and $E(t)$ is the set of edges that are present between nodes in $U$ and nodes in $V$ at time $t$. The number of edges $|E(\cdot)|$ changes over time and can vary from a minimum of $0$ to a maximum of $MN$. We set $G(0) = G$, where $G$ is the initial bipartite graph. We denote by $G_{MN} = (U \sqcup V, E_{MN})$ the complete bipartite graph associated to $(U,V)$ and, for every edge $e \in E_{MN}$, at time $t$ we define the Bernoulli random variable $Y_e(t)$ as
\begin{equation}
Y_e(t) =
\begin{cases}
0, \text{ if } e \notin E(t), \\
1, \text{ if } e \in E(t).
\end{cases}
\end{equation}
In other words, $Y_e(t) = 0$ if edge $e$ is not present in the graph at time $t$, while $Y_e(t) = 1$ if it is present. The \textit{joint edge activity state} at time $t$ is denoted by
\begin{equation}
{\bf Y}(t) = \{ Y_e(t) \}_{e \in E_{MN}}
\end{equation}
and is an element of the state space 
\begin{equation}
\mathcal{Y} = \big\{{\bf Y} \in \{0,1\}^{MN} \big\}.
\end{equation}
The degree of node $v$ at time $t$ is denoted by $d_v(t)$. 
\end{definizione}

We model the dynamics of the graph in the following way. If an edge is not present, then it \textit{appears} according to a Poisson clock with rate $\lambda$, independently of the other edges. If an edge is present, then it \textit{disappears} according to a Poisson clock with rate $\lambda$, independently of the other edges. This is equivalent to having a collection of i.i.d.\ Poisson clocks with rate $\lambda$ on the edges and letting an edge change its state every time its clock ticks. In order to study how the edge dynamics affect the transition time, we consider Poisson clocks with rates $\lambda = \lambda(r)$ depending on the parameter $r$.

Throughout the paper we use the notation $\prec$, $\preceq$, $\asymp$, $\succeq$, $\succ$ to describe the asymptotic behavior in the limit $r \to \infty$. More precisely, as $r \to \infty$, $f(r) \prec g(r)$ means that $f(r) = o(g(r))$, $f(r) \preceq g(r)$ means that $f(r) = \mathcal{O}(g(r))$, $f(r) \asymp g(r)$ means that $f(r) = \Theta(g(r))$, $f(r) \succeq g(r)$ means that $g(r) = \mathcal{O}(f(r))$, and $f(r) \succ g(r)$ means that $g(r) = o(f(r))$.

\begin{remark}[\bf{Model variations}]
\label{modelvariations}
Below we discuss some model variations that we do not consider in this paper but might lead to other interesting scenarios.
\begin{itemize}
\item[(i)] {\bf Edge rates}. We may allow different rates for the edges to change their state. Denote by $\lambda^+(r)$ and $\lambda^-(r)$ the rates at which edges appear and disappear, respectively. If these are of the same order, then we are in a situation similar to them being equal to $\lambda(r)$. If $\lambda^+(r) \to \infty$ and $\lambda^-(r) \prec \lambda^+(r)$, then, with high probability as $r \to \infty$, in time $o(1)$ the dynamics turn the initial graph into the complete bipartite graph with all the edges present. Analogously, if $\lambda^-(r) \to \infty$ and $\lambda^-(r) \succ \lambda^+(r)$, then, with high probability as $r \to \infty$, in time $o(1)$ the dynamics turn the initial graph into the empty graph with all the edges absent. Both these assumptions do not lead to interesting models. However, when $\lambda^+(r)$ and $\lambda^-(r)$ are of a different order and do not tend to infinity, we have an intermediate situation where at any time $t$ an edge is either present with high probability as $r \to \infty$ or absent with high probability as $r \to \infty$, but the total amount of time the edge has been absent or present, respectively, up to time $t$ is not always negligible.
\item[(ii)] {\bf Appearing edge}. When an edge appears between two active nodes, we assume that the active node in $U$ deactivates, since the model does not allow two connected nodes to be simultaneously active. It would be interesting to study also alternative models, where the active node in $V$ deactivates or where the deactivating node is chosen at random. The fact that these models slow down the transition motivates our choice. 
\end{itemize}
\end{remark}

\subsection{The graph evolution}
\label{sec:graphevolution}

The node activity process $({\bf X}(t), {\bf Q}(t))_{t \geq 0}$ and the edge activity process $({\bf Y}(t))_{t \geq 0}$ form a continuous-time Markov process on 
\begin{equation}
\mathcal{X} \times \mathbb{R}_{\geq 0} \times \mathcal{Y}
\end{equation}
that describes the evolution of the graph under the effects of the dynamics. We refer to this process as the \textit{graph evolution process}. Control on this process allows us to understand how the degrees of the nodes change over time and how nodes activate, and it is the key to obtaining precise asymptotics for the mean transition time.

In order to keep track of the states of the nodes and the edges in the graph, we focus on the number of active neighbors each node has.

\begin{definizione}[\bf {Active degree}]
We define the \textit{active degree} of a node as the number of its active neighbors. 
For $u \in U$, the active degree at time $t$ is given by
\begin{equation}
\label{activedegreeu}
\tilde{d}_u(t) = | \{v \in V:\, uv \in E(t), X_v(t) =1 \} |.
\end{equation}
Analogously, for $v \in V$, the active degree at time $t$ is given by 
\begin{equation}
\label{activedegreev}
\tilde{d}_v(t) = | \{u \in U:\, uv \in E(t), X_u(t) =1 \} |.
\end{equation}
\end{definizione}

\noindent
Note that for a node to activate, its active degree must be 0. It is immediate to see that the active degree of a node cannot exceed its degree, i.e., for any $u \in U$ and $v \in V$
\begin{equation}
\tilde{d}_u(t) \leq d_u(t) \qquad \text{and} \qquad \tilde{d}_v(t) \leq d_v(t).
\end{equation}

Consider the set of feasible states where some nodes are active and some edges are present, where by feasible we mean that it respects the constraints given by the edges, for which two connected nodes cannot be active simultaneously. An arbitrary feasible state at time $t$ has $h$ active nodes in $U$ with $h= 0, \ldots, M$, $k$ active nodes in $V$ with $k=0, \ldots, N$, and $l$ present edges with $l=0, \ldots, MN$. Consequently, there are $M-h$ inactive nodes in $U$, $N-k$ inactive nodes in $V$, and $MN-l$ absent edges. Note that the initial state ${\bf 1_U}$ is described by $h=M$, $k=0$ and $l=|E(0)|$, while the transition occurs as soon as state ${\bf 1_V}$ is reached, for which $k=N$. The graph evolution is governed by different Poisson clocks ticking at various rates: the activation clocks, the deactivation clocks and the edge clocks (note that the queue lengths, hence the input process, also play a role, since the activation rates depend on them). We analyze how the graph evolves each time one of these clocks ticks.

\begin{itemize}
\item[-] The activation clock of a node $u \in U$ ticks at rate $g_U(Q_u(t))$ at time $t$. The probability of this clock being the first one to tick is given by
\begin{equation}
\frac{g_U(Q_u(t))}{Z},
\end{equation}
with 
\begin{equation}
Z = \sum_{i=1}^{M-h} g_U(Q_i(t)) + \sum_{j=1}^{N-k} g_V(Q_j(t)) + h + k + MN \lambda(r).
\end{equation}
The tick has two possible effects: if the neighbors of $u$ are all inactive, then $u$ activates and the active degrees of all its neighbors increase by $1$; if there is at least one active neighbor of $u$, then the activation attempt fails and nothing happens.
\item[-] The deactivation clock of a node $u \in U$ ticks at rate $1$. The probability of this clock being the first one to tick is given by $\frac{1}{Z}$. Node $u$ deactivates and the active degrees of all its neighbors decrease by $1$.
\item[-] The activation clock of a node $v \in V$ ticks at rate $g_V(Q_v(t))$ at time $t$. The probability of this clock being the first one to tick is given by $\frac{g_V(Q_v(t))}{Z}$.
The tick has two possible effects: if the neighbors of $v$ are all inactive, then $v$ activates and the active degrees of all its neighbors increase by $1$; if there is at least one active neighbor of $v$, then the activation attempt fails and nothing happens.
\item[-] The deactivation clock of a node $v \in V$ ticks at rate $1$. The probability of this clock being the first one to tick is given by $\frac{1}{Z}$. Node $v$ deactivates and the active degrees of all its neighbors decrease by 1.
\item[-] The activation clock of an edge $e \in E_{MN}$ ticks at rate $\lambda(r)$. The probability of this clock being the first one to tick is given by $\frac{\lambda(r)}{Z}$. Depending on which edge appears or disappears and on the nodes involved, the tick has different effects on the graph, which are described below. 
\end{itemize}

If we know the number of active nodes in $U$ and $V$, then we can compute the probabilities of each of the following scenarios with simple combinatorial arguments.
There are four possible scenarios in which an edge can appear.
\begin{itemize}
\item[($\circ \, \circ$)] When an edge between two inactive nodes appears, their degrees increase by 1.
\item[($\circ \, \bullet$)] When an edge between an inactive node in $U$ and an active node in $V$ appears, the active degree of the node in $U$ increases by 1 and the degree of the node in $V$ increases by 1.
\item[($\bullet \, \circ$)] When an edge between an active node in $U$ and an inactive node in $V$ appears, the degree of the node in $U$ increases by 1 and the active degree of the node in $V$ increases by 1.
\item[($\bullet \, \bullet$)] When an edge between two active nodes appears, the node in $U$ deactivates, its active degree increases by 1, the active degrees of all its neighbors in $V$ decrease by 1 and the degree of the node in $V$ increases by 1.
\end{itemize}
In a similar fashion, there are three possible scenarios in which an edge can disappear. Recall that there cannot be an edge between two active nodes.
\begin{itemize}
\item[($\circ \, \circ$)] When an edge between two inactive nodes disappears, their degrees decrease by 1.
\item[($\circ \, \bullet$)] When an edge between an inactive node in $U$ and an active node in $V$ disappears, the active degree of the node in $U$ decreases by 1 and the degree of the node in $V$ decreases by 1.
\item[($\bullet \, \circ$)] When an edge between an active node in $U$ and an inactive node in $V$ disappears, the degree of the node in $U$ decreases by 1 and the active degree of the node in $V$ decreases by 1.
\end{itemize}

Note that the transition time is strongly related to the graph evolution process, since the activation times of the nodes in $V$ depend on the activation rates, the speed of the dynamics and the active degrees of the nodes.

\section{Main results}
In this section we state the main results of the paper. We first consider the queue-based model described in Section~\ref{modeldescription} and we show how the dynamics affect the order of the mean transition time (Theorem~\ref{meantrtime}). We then consider a simplified version of the model where the activation rates are fixed and we adapt results from \cite{BdHNS18} and \cite{BdHNS20} in order to study the mean transition time of networks modeled by complete bipartite graphs (Theorem~\ref{adaptedthmpaper1}), arbitrary bipartite graphs (Theorem~\ref{adaptedthmpaper2}) and dynamic bipartite graphs (Theorem~\ref{adaptedmeantrtime}).

\subsection{Model with queue-based activation rates}

We begin with an overview of the main definitions and results in \cite{BdHNS20}, which studies the transition time of static networks modeled by arbitrary bipartite graphs with queue-based activation rates. These results are important in order to extend the analysis of the transition time to networks modeled by dynamic bipartite graphs.

The key factor in \cite{BdHNS20} is the introduction of a \textit{randomized algorithm} that takes the graph as input and gives as output the set of transition paths the network is most likely to follow. Along each path it is possible to determine the mean transition time and its law on the scale of its mean. 
A \textit{path} $a = (v_1, \dots, v_N)$ is defined as a sequence of activating nodes in $V$, where each node is present only once. The set of all paths, i.e., of every possible ordering of the nodes in $V$, is denoted by $\Omega$. The algorithm takes as input the graph and gives as output a path together with its probability: at every step $k=1, \dots, N$, each of the $n_k$ nodes in $V_k$ of minimum degree $\bar{d}_k$ is selected with probability $1/n_k$, where $G_k = (U_k, V_k)$ is the induced subgraph of $G$ in which the first $k-1$ selected nodes in $V$ and their neighbors are removed. The set of all paths that can be generated by the algorithm is referred to as the set of \textit{admissible paths} and it is denoted by $\mathcal{A}$. 
With each selected node, it is associated the \textit{nucleation time} of the complete bipartite subgraph consisting of only the node and its neighbors (see Definition~\ref{nucleationdef}), and it is known that the leading order term of the mean nucleation time depends on the degree of the node. It is possible to compute the transition time along a path, by decomposing the transition into the succession of nucleations of the nodes in the path. It then becomes crucial to determine the maximum degree along an admissible path.

\begin{definizione}[\bf{Maximum least degree}]
Given the sequence $(\bar{d}_k)_{k=1}^N$ generated by the algorithm, let 
\begin{equation}
\label{d*}
d^* = \max_{1 \leq k \leq N} \bar{d}_k
\end{equation}
be the \textit{maximum least degree} of the admissible path associated with $(\bar{d}_k)_{k=1}^N$.
\end{definizione}
It is shown in \cite[Proposition 2.9]{BdHNS20} that the algorithm is greedy, in the sense that at every step it selects the node that adds the least to the transition time along the path. It is also shown in \cite[Proposition 2.10]{BdHNS20} that the algorithm is consistent, in the sense that all the admissible paths have the same value of $d^*$. Hence, the maximum least degree is unique and it determines the leading order of the mean transition time along any admissible paths. 

The main results for static bipartite graphs are proven in \cite[Theorems 3.3 and 3.5]{BdHNS20} and show that, depending on the relation between $d^*$ and $\beta$, the transition exhibits a trichotomy: we observe a phase transition which we describe with the terms \textit{subcritical regime}, \textit{critical regime} and \textit{supercritical regime}, borrowed from statistical physics. If $\beta \in (0, \frac{1}{d^*-1})$, we are in the subcritical regime where each node in $V$ activates in a random time described by an exponential variable with mean $o(r)$. If $\beta = \frac{1}{d^*-1}$, we are in the critical regime where some (or all) nodes in $V$ activate in a random time described by a variable with truncated polynomial law and mean of order $r$. In those two regimes, the mean transition time of the graph can be computed by averaging over all the admissible paths. If $\beta \in (\frac{1}{d^*-1}, \infty)$, we are in the supercritical regime where nodes in $U$ are so aggressive that the time it takes to activate at least one node in $V$ is given by the average time it takes for the queue lengths at nodes in $U$ to hit zero (which is of order $r$), hence a deterministic behavior observed, also known in the literature as cut-off. In this regime, the mean transition time is then the time it takes on average for the queue lengths at nodes in $U$ to hit zero.
 
Our goal in this paper is to extend the above-mentioned results of \cite{BdHNS20} for static bipartite graphs to dynamic bipartite graphs. In order to achieve it, we distinguish between different types of dynamics and study how they affect the mean transition time. For static bipartite graphs, it is also possible to identify the law of the transition time divided by its mean (in the subcritical and supercritical regimes). However, this goes beyond the scope of the present paper, since understanding the effects of the edge dynamics is rather challenging. 

Through the paper, we write $\mathbb{P}_{1_U}$ and $\mathbb{E}_{1_U}$ to denote probability and expectation on the path space given that the initial configuration is ${\bf 1_U}$.

\begin{remark}[{\bf Typical behavior of the network}]
\label{typical}
Note that we are interested in the \textit{typical} behavior of the network, in the following sense.
\begin{itemize}
\item[(i)]
Let $\mathcal{F}$ be the event that the network follows the algorithm, i.e., that the transition follows any of the admissible paths. The mean transition time of the graph $G$ given the initial queue lengths ${\bf Q^0}$ can be then split as
\begin{equation}
\label{eqsplittransition}
\mathbb{E}_{1_U}[\mathcal{T}_{G}^{Q^0}] = \mathbb{E}_{1_U}[\mathcal{T}_{G}^{Q^0} \mathbbm{1}_{\mathcal{F}}]  
+ \mathbb{E}_{1_U}[\mathcal{T}_{G}^{Q^0}\mathbbm{1}_{\mathcal{F}^C}].
\end{equation}
The second term in the right-hand side represents the mean transition time when the network does \emph{not} follow the algorithm, and equals
\begin{equation}
\mathbb{E}_{1_U}[\mathcal{T}_{G}^{Q^0}\mathbbm{1}_{\mathcal{F}^C}] 
=\mathbb{E}_{1_U}[\mathcal{T}_{G}^{Q^0} | \mathcal{F}^C] \, \mathbb{P}_{1_U}(\mathcal{F}^C).
\end{equation}
It is proved in \cite[Theorem 3.2]{BdHNS20} that $\lim_{r \to \infty} \mathbb{P}_{1_U}(\mathcal{F}^C) =0$. However, the total mean transition time can still be affected by this term, since the conditional expectation may be substantial. The work in \cite{BdHNS20} focuses on the first term in the right-hand side of \eqref{eqsplittransition}, which then captures the typical behavior of the network.

\item[(ii)] When introducing the dynamics, we will see in Proposition~\ref{NvsD} that the activation of each node in $V$ is characterized by a competition between its nucleation and the disconnection from its neighbors due to the dynamics. In many cases, we are able to describe the outcome of this competition with high probability as $r \to \infty$. We then condition on an event $\mathcal{B}$ such that $\lim_{r \to \infty} \mathbb{P}_{1_U} (\mathcal{B}^C) = 0$, which captures the typical behavior of the competition. More details can be found in the proofs in Section~\ref{proofs}.
\end{itemize}
Our results are then conditioned on the event $\mathcal{E} = \mathcal{F} \cap \mathcal{B}$, which captures the typical behavior of the network with respect to the transition path it follows and to the outcome of the competition for the activation of the nodes in $V$. Note that $\mathcal{E}$ is intended to capture the typical behavior of the network only until time scale $\lambda(r)^{-1}$, which turns out to be the time scale of the mean transition time when the dynamics are ``not too slow". In particular, in the case of competitive dynamics (SDc), when we consider $\lambda(r) = r^{-\alpha}$ with $0 < \alpha \leq 1 \wedge \beta(d^*-1)$, the event $\mathcal{E}(\alpha) = \mathcal{F}(\alpha) \cap \mathcal{B}(\alpha)$ refers to the typical behavior until time scale $r^{\alpha}$.
\end{remark}

Below, we state our main result. We denote by $\mathcal{T}_{G(\cdot)}^{Q^0}$ the transition time of the dynamic graph $G(\cdot)$ conditional on the initial queue lengths ${\bf Q^0}$.

\begin{teorema}[\bf {Mean transition time for dynamic bipartite graphs}]
\label{meantrtime}
Consider the dynamic bipartite graph $G(\cdot)= ((U,V),E(\cdot))$ with the edge dynamics governed by $\lambda(r)$ and initial queue lengths ${\bf Q^0}$. Suppose that Assumption~\ref{queuedependentrates} holds.

\begin{itemize}
\item[$(\mathrm{FD})$] If the dynamics are fast, i.e., $\lambda(r) \to \infty$, then the transition time satisfies
\begin{equation}
\mathbb{E}_{1_U}[\mathcal{T}_{G(\cdot)}^{Q^0} \, | \, \mathcal{E}] \asymp \lambda(r)^{-1} = o(1), \qquad r \to \infty.
\end{equation}
\item[$(\mathrm{RD})$] If the dynamics are regular, i.e., $\lambda(r) = C \in (0, \infty)$, then the transition time satisfies
\begin{equation}
\mathbb{E}_{1_U}[\mathcal{T}_{G(\cdot)}^{Q^0}\, | \, \mathcal{E}] \asymp \lambda(r)^{-1}= \mathcal{O}(1), \qquad r \to \infty.
\end{equation}
\item[$(\mathrm{SD})$] If the dynamics are slow, i.e., $\lambda(r) \to 0$, then the following cases occur.
\begin{itemize}
\item[$(\mathrm{SDc})$] If the dynamics are competitive, i.e., $\lambda(r) \succeq r^{-(1 \wedge \beta(d^*-1))}$, then the transition time satisfies
\begin{equation}
\mathbb{E}_{1_U}[\mathcal{T}_{G(\cdot)}^{Q^0}\, | \, \mathcal{E}] \asymp \lambda(r)^{-1}, \qquad r \to \infty.
\end{equation}
More precisely, let $\lambda(r) = r^{-\alpha}$ with $0 < \alpha \leq 1 \wedge \beta(d^*-1)$, and let $T_U(r) = \frac{\gamma_U}{c-\rho_U} r \, [1+o(1)]$ be the average time it takes for the queue lengths at nodes in $U$ to hit zero. Then the transition time satisfies
\begin{equation}
\mathbb{E}_{1_U}[\mathcal{T}_{G(\cdot)}^{Q^0} \, | \, \mathcal{E}(\alpha)] \asymp r^{\alpha}, \qquad r \to \infty.
\end{equation}
Moreover, for $\alpha=1$, when $\beta = \frac{1}{d^*-1}$, i.e., in the critical regime,
\begin{equation}
\lim_{r \to \infty} \mathbb{P}_{1_U} \big(\mathcal{T}_{G(\cdot)}^{Q^0} < T_U(r) \, | \, \mathcal{E}(1) \big)=1,
\end{equation}
while when $\beta \in (\frac{1}{d^*-1}, \infty)$, i.e., in the supercritical regime,
\begin{equation}
\lim_{r \to \infty} \mathbb{P}_{1_U} \big(\mathcal{T}_{G(\cdot)}^{Q^0} = T_U(r) \,|\, \mathcal{B}(1)\big) = 1- \lim_{r \to \infty} \mathbb{P}_{1_U} \big(\mathcal{T}_{G(\cdot)}^{Q^0} < T_U(r) \,|\, \mathcal{B}(1)\big) > 0.
\end{equation}

\item[$(\mathrm{SDnc})$] If the dynamics are non-competitive, i.e., $\lambda(r) \prec r^{-(1 \wedge \beta(d^*-1))}$, then, conditional on the event $\mathcal{E}$, Theorem 3.3 in \cite{BdHNS20} holds.
\end{itemize}
\end{itemize}
\end{teorema}

Note that the order of the mean transition time depends on the speed of the dynamics. When the dynamics are fast (FD), the edges quickly appear and disappear, reaching in time $o(1)$ the state where nodes in $V$ have no edges connecting them to $U$. Since nodes in $V$ are aggressive, they eventually activate in time $o(1)$. When the dynamics are regular (RD), the situation is similar, but it takes time $\mathcal{O}(1)$ to reach the state where all the edges are simultaneously absent. When the dynamics are slow (SD), in the case of competitive dynamics (SDc), the relation between the speed of the dynamics and the aggressiveness of the nodes in $U$ plays a key role, while in the case of non-competitive dynamics (SDnc), the network behaves as if the edges were fixed at the initial configuration and there were no dynamics. Note that, as we anticipated, in the cases of fast, regular and competitive dynamics, the order of the mean transition time is given by the reciprocal of the rate $\lambda(r)$. 
The term \textit{competitive} is chosen to describe the case in which some nodes in $V$ activate with positive probability \textit{both} because their neighbors are simultaneously inactive and because the edges connecting them to their neighbors disappear.

\subsection{Model with fixed activation rates}
\label{modelwithfixedrates}
We have seen how the dynamics affect the mean transition time of wireless random-access models where the activation rates depend on the current queue lengths at the nodes. Not much is known in the literature for random-access protocols with dynamic interference graphs, even for models with \textit{fixed activation rates}. In this section we adapt the theory built in \cite{BdHNS18} and \cite{BdHNS20} to study the effects of the dynamics on these types of models. 

\begin{ass}[{\bf Fixed activation rates}]
\label{fixedrates}
Assume that the deactivation clocks tick at rate $1$, while the activation clocks tick at rate
\begin{equation}
\phi_w(t) = 
\left\{\begin{array}{ll} 
r^{\beta},& \text{ if } w \in U, \\
r^{\beta'},& \text{ if } w \in V,
\end{array}
\right.
\qquad t \geq  0.
\end{equation}
with $\beta, \beta' \in (0, \infty)$ and $\beta' > \beta + 1$. Recall that we are interested in the transition time asymptotics as $r \to \infty$.
\end{ass}

We start by adapting the result for complete bipartite graphs from \cite[Theorem 1.7]{BdHNS18} to the model with fixed activation rates. The following theorem is consistent with \cite[Example 4.1]{dHNT18}.

\begin{teorema}[{\bf Mean transition time for complete bipartite graphs}]
\label{adaptedthmpaper1}
Consider the complete bipartite graph $G = ((U,V),E)$ with initial queue lengths ${\bf Q^0}$ as in \eqref{initialqueues}. Suppose that Assumption~\ref{fixedrates} holds.
\begin{itemize}
\item[{\rm (I)}] $\beta \in (0, \frac{1}{|U|-1})$: subcritical regime. The transition time satisfies
\begin{equation}
\mathbb{E}_{1_U}[\mathcal{T}_G^{Q^0}] = \frac{1}{|U|} r^{\beta (|U|-1)} \, [1+o(1)], \qquad r \to \infty.
\end{equation}
\item[{\rm (II)}] $\beta = \frac{1}{|U|-1}$: critical regime. The transition time satisfies
\begin{equation}
\mathbb{E}_{1_U}[\mathcal{T}_G^{Q^0}] = \frac{1}{|U|} r \, [1+o(1)], \qquad r \to \infty.
\end{equation}
\item[{\rm (III)}] $\beta \in (\frac{1}{|U|-1}, \infty)$: supercritical regime. The transition time satisfies
\begin{equation}
\mathbb{E}_{1_U}[\mathcal{T}_G^{Q^0}] = \frac{\gamma_U}{c-\rho_U} r \, [1+o(1)], \qquad r \to \infty.
\end{equation}
\end{itemize}
\end{teorema}

Next, we adapt the result from \cite[Theorem 3.3]{BdHNS20} to arbitrary bipartite graphs with fixed activation rates. Note that the algorithm still plays a crucial role in determining the mean transition time. Recall that, at each step $k=1, \dots, N$ of the algorithm, $n_k$ denotes the number nodes in $V_k$ of minimum degree $\bar{d}_k$.

\begin{teorema}[{\bf Mean transition time for arbitrary bipartite graphs}]
\label{adaptedthmpaper2}
Consider the bipartite graph $G = ((U,V),E)$ with initial queue lengths ${\bf Q^0}$ as in \eqref{initialqueues}. Suppose that Assumption~\ref{fixedrates} holds. Let $\mathcal{F}_a$ be the event that the network follows the admissible path $a \in \mathcal{A}$.
\begin{itemize}
\item[{\rm (I)}] $\beta \in (0, \frac{1}{d^*-1})$: subcritical regime. The transition time satisfies
\begin{equation}
\mathbb{E}_{1_U}[\mathcal{T}_G^{Q^0} \, | \, \mathcal{F}_a] = \sum_{\substack{1 \leq k \leq N \\ k: \,\bar{d}_k = d^*}} \frac{1}{n_k d^*}\, r^{\beta(d^*-1)}\,[1+o(1)], \qquad r \to \infty.
\end{equation}
\item[{\rm (II)}] $\beta = \frac{1}{d^*-1}$: critical regime. Then the transition time satisfies
\begin{equation}
\mathbb{E}_{1_U}[\mathcal{T}_G^{Q^0} \, | \, \mathcal{F}_a] = \sum_{\substack{1 \leq k \leq N \\ k: \,\bar{d}_k = d^*}} \frac{1}{n_k d^*}\, r\,[1+o(1)], \qquad r \to \infty.
\end{equation}
The above result holds as long as the pre-factor is below the value $\frac{\gamma_U}{c-\rho_U}$, which corresponds to the time it takes for the queue lengths at nodes in $U$ to hit zero. Otherwise, the supercritical regime applies.
\item[{\rm (III)}] $\beta \in (\frac{1}{d^*-1},\infty)$: supercritical regime. The transition time satisfies
\begin{equation}
\mathbb{E}_{1_U}[\mathcal{T}_G^{Q^0}] = \frac{\gamma_U}{c-\rho_U} r\, [1+o(1)], \qquad r \to \infty.
\end{equation}
\end{itemize}
\end{teorema}

Finally, we show that the results from Theorem~\ref{meantrtime} also hold when we consider a dynamic bipartite graph with fixed activation rates. We are able to compute the order of the mean transition time, while the pre-factor still depends on the graph evolution described in Section~\ref{sec:graphevolution}.

\begin{teorema}[{\bf Mean transition time for dynamic bipartite graphs}]
\label{adaptedmeantrtime}
Consider the dynamic bipartite graph $G(\cdot) = ((U,V),E(\cdot))$ with the edge dynamics governed by $\lambda(r)$ and initial queue lengths ${\bf Q^0}$. Suppose that Assumption~\ref{fixedrates} holds. Then the results of Theorem~\ref{meantrtime} hold.
\end{teorema}

\subsection{Discussion and related work}
\label{sectiondiscussion}

{\bf Intuition}. The intuition behind the results is that a node in $V$ can activate for two reasons. It can activate when its neighbors are simultaneously inactive or when there are no edges connecting it to nodes in $U$. Interpolation between these two situations gives rise to different cases, which mainly depend on the speed of the dynamics. In the case of competitive dynamics, we are able to distinguish between different behaviors for the mean transition time by analyzing the subcritical, critical and supercritical regimes separately. To summarize, with high probability as $r \to \infty$, the order of activation of nodes in $V$ follows an admissible path until the edge dynamics of rate $\lambda(r)$ become competitive. The competition begins on time scale $\lambda(r)^{-1}$, the time scale on which all the remaining nodes in $V$ activate, if there are any, and the transition occurs.\\

\noindent
{\bf Assumptions}.
The assumption $\beta' > \beta + 1$ in Assumptions~\ref{queuedependentrates} and \ref{fixedrates} ensures that the transition on an arbitrary bipartite graph can be decomposed into a succession of transitions on complete bipartite subgraphs. Indeed, it has been shown in \cite[Lemma 2.2]{BdHNS20} that when a node in $V$ activates, it "blocks`` all its neighbors in $U$, in the sense that with high probability they will remain inactive for the rest of the time. Consider an inactive node $u \in U$ and note that it could activate again only when all its neighbors in $V$ are simultaneously inactive. In the worst-case scenario where $u$ has only one active neighbor $v$, they will compete with each other for the entire duration of the transition. Our assumption guarantees that with high probability $u$ will never win any competition against $v$, and hence will remain blocked. If we only assume $\beta' > \beta$, there could be nodes in $U$ activating again over time, which would lead to a much more challenging model to analyze.\\

\noindent
{\bf Activation rates and related work}.
Most of the literature refers to models where the activation rates are fixed parameters and the underlying Markov process is time-homogeneous (see \cite{dHNT18}, \cite{NZB16}, \cite{Z17}, \cite{Z18}). In this setting it has been shown that the transition time from the metastable to the stable state is approximately exponential on the scale of its mean. The main idea is to consider the return times to the metastable state of the discrete-time embedded Markov chain as regeneration times. At each regeneration time a Bernoulli trial is conducted. The trial is successful if the stable state is reached before a return to the metastable state occurs, while it is unsuccessful otherwise. In the asymptotic regime, the success probability of each trial is small and the expected length of a single trial is negligible compared to the expected transition time. It is known that the first success time of a large number of trials, each having a small probability of success, is approximately exponentially distributed (see \cite{FMNS15}, \cite{K79}).

Attention has also been paid to models where the activation rates are deterministic functions of time. The underlying Markov process is therefore time-inhomogeneous, and it has been shown that, under appropriate conditions, the transition time is approximately exponential with a non-constant rate (see \cite{BdHNT19}). The above-described approach is still fruitful, but the success probability and the length of each trial now depend on its
starting time. 

Recently, various models for random-access networks with queue-based protocols have been investigated (see \cite{B13}, \cite{CBMSW21}, \cite{DP18}, \cite{DP19}, \cite{NTS10}, \cite{SS10}). Note that, since the transition rates depend on time only via the current state of the vector, $({\bf X}(t), {\bf Q}(t))_{t \geq 0}$ for static networks and $({\bf X}(t), {\bf Q}(t), {\bf Y}(t))_{t \geq 0}$ for dynamic networks are time-homogeneous Markov processes with state space $\mathcal{X} \times \mathbb{R}_{\geq 0}^N$ and $\mathcal{X} \times \mathbb{R}_{\geq 0}^N \times \mathcal{Y}$, respectively. The state-dependent nature of the activation rates creates challenging problems, whose solution requires a new methodological approach in order to analyze how the node activity interacts with the queue lengths and eventually with the edge dynamics. The stationary distributions of these Markov processes, in general, do not admit a closed-form expression and even the basic throughput characteristics and stability conditions are not in general known. It is not simple to describe explicitly the stability condition for general network topologies (see \cite{vdVBvLP10}) and only structural representations or asymptotic results are known (see \cite{CBvL14}, \cite{LK13}). For suitable activation rate functions, it has been shown that queue-based protocols are capable to achieve maximum stability and the optimal throughput performance of centralized scheduling algorithms (see \cite{GS10}, \cite{JSSW10}, \cite{RSS09}, \cite{SS12}). However, they tend to lead to long queues and delays (see \cite{BCvL14}, \cite{GBW14}). For these reasons, analyzing the delay performance of queue-based random-access protocols has recently attracted great attention.


\section{The edge dynamics}
In this section we analyze the time it takes for the nodes in $V$ to be disconnected from $U$ and we show how different types of dynamics can slow down or speed up their activation. The goal of the section is Proposition~\ref{NvsD}, which shows how the activation of each node in $V$ is characterized by the competition between its nucleation and the effects of the dynamics. Proposition~\ref{NvsD} is crucial to prove our main results in Section~\ref{proofs}.

\subsection{Disconnection time}

Due to the dynamics, a node $v \in V$ can activate if at some point there are no edges connecting it to nodes in $U$. Indeed, the dynamics might take the graph to a configuration where the degree of $v$ is temporarily $0$, so that $v$ can activate as soon as its clock ticks.

\begin{definizione}[\bf {Disconnection time}]
Given $v \in V$, we call \textit{disconnection time} of $v$ the time it takes for $v$ to be disconnected from $U$, i.e., to have all possible edges connecting it to $U$ simultaneously absent. We denote by $D_v^{Q^0}$ the disconnection time of $v$ conditional on the initial queue lengths ${\bf Q^0}$.
\end{definizione}

As introduced in Section~\ref{ss:model}, the dynamics affect the graph by allowing the edges to appear and disappear according to a Poisson clock with rate $\lambda(r)$. The alternation between the states of each edge  $e \in E_{MN}$ is described by an exponential random variable $S_e \simeq \mathrm{Exp}(\lambda(r))$ with mean $\mu(r) = \lambda(r)^{-1}$. Note that, with high probability as $r \to \infty$, $S_e$ takes values of the order of its mean, i.e., $S_e \asymp \mu(r)$. Indeed, if we pick $x \prec \mu(r)$, then
\begin{equation}
\label{wn3:Deorder}
\lim_{r \to \infty} \mathbb{P} (S_e \leq x) = \lim_{r \to \infty} 1-e^{-\lambda(r) x} = 0, 
\end{equation}
and the same holds for $x \succ \mu(r)$. In other words, if an edge is absent at time $t$, then, with high probability as $r \to \infty$, it will take an amount of time of order $\mu(r)$ for the Poisson clock to tick and for the edge to become present. Vice versa, if an edge is present at time $t$, then it will take an amount of time of order $\mu(r)$ for the edge to become absent.

The arbitrary bipartite initial configuration of the graph plays an important role in understanding the transition time. Consider a node in $v \in V$ of initial degree $d_v(0)  = d >0$. Since $|U| = M$, there are $M$ possible total edges connecting $v$ to $U$. We construct a continuous-time Markov chain $\mathcal{M}$ where each state $k$ represents the set of configurations of the $M$ edges in which $k$ edges are present and $M-k$ edges are absent. State 0 corresponds to all edges being absent, state 1 corresponds to the $M$ possible configurations with exactly one edge present, and so on (see Figure~\ref{fig:Mc} below).

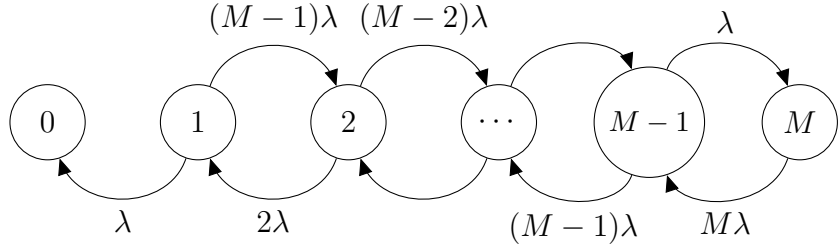
\begin{figure}[htbp]
\begin{center}
\begin{tikzpicture}[>=triangle 45, scale = 0.8]
 \node (zero) [draw, circle, minimum size=1cm] at (-2.5,1cm) {$0$};
 \node (one) [draw, circle, minimum size=1cm] at (0,1cm) {$1$};
 \node (two) [draw, circle, minimum size=1cm] at (2.5,1cm) {$2$};
 \node (dots) [draw, circle, minimum size=1cm] at (5,1cm) {$\cdots$};
 \node (emminusone) [draw, circle, minimum size=1cm] at (7.5,1cm) {\small{$M-1$}};
 \node (en) [draw, circle, minimum size=1cm] at (10,1cm) {$M$};
 \draw [->] (one) .. controls +(0.5,1.5) and +(-0.5,1.5) .. node [midway, above] {$(M-1) \lambda$} (two);
 \draw [->] (two) .. controls +(0.5,1.5) and +(-0.5,1.5) .. node [midway, above] {$(M-2) \lambda$} (dots);
  \draw [->] (dots) .. controls +(0.5,1.5) and +(-0.5,1.5) .. node [midway, above] {} (emminusone);
 \draw [->] (emminusone) .. controls +(0.5,1.5) and +(-0.5,1.5) .. node [midway, above] {$\lambda$} (en);
 \draw [->] (en) .. controls +(-0.5,-1.5) and +(0.5,-1.5) .. node [midway, below] {$M \lambda$} (emminusone);
 \draw [->] (emminusone) .. controls +(-0.5,-1.5) and +(0.5,-1.5) .. node [midway, below] {$(M-1) \lambda$} (dots);
  \draw [->] (dots) .. controls +(-0.5,-1.5) and +(0.5,-1.5) .. node [midway, below] {} (two);
 \draw [->] (two) .. controls +(-0.5,-1.5) and +(0.5,-1.5) .. node [midway, below] {$2 \lambda$} (one);
 \draw [->] (one) .. controls +(-0.5,-1.5) and +(0.5,-1.5) .. node [midway, below] {$\lambda$} (zero);
\end{tikzpicture}
\end{center}
\caption{{\small The Markov chain $\mathcal{M}$ describing how the edge dynamics governed by $\lambda = \lambda(r)$ change the degree of a node in $V$. It is a birth-death process with $M$ transient states and one absorbing state.}}
\label{fig:Mc}
\end{figure}

We consider state 0 as an absorbing state, since we are interested in computing the hitting times to state 0 starting from any other state. From state $M$ we can only jump to state $M-1$, when one of the $M$ present edges disappears, which happens with rate $M \lambda(r)$. From each state $0 < k < M$ we jump to the neighboring states also with rate $M \lambda(r)$. Indeed, as soon as the clock of one of the $M$ possible edges ticks, we jump to the state $k+1$ if the edge was absent and becomes present, while we jump to the state $k-1$ if the edge was present and becomes absent. Hence, we jump from state $k$ to state $k+1$ with probability $\frac{M-k}{M}$, while we jump from state $k$ to state $k-1$ with probability $\frac{k}{M}$. 

The transition rate matrix $H(r)$ of the Markov chain $\mathcal{M}$ is given by
\renewcommand{\kbldelim}{(}
\renewcommand{\kbrdelim}{)}
\begin{equation}
H(r)= \kbordermatrix{
      & 0 & 1 & 2 & \cdots & M-1 & M \\
   0 & 0 & 0 & 0 & 0 & 0 & 0 \\
   1 & \lambda(r) & -M\lambda(r)  & (M-1) \lambda(r)  & 0 & 0 & 0 \\
   2 & 0 & 2 \lambda(r) & -M \lambda(r) & \cdots & 0 & 0 \\
   \vdots & 0 & 0 & \cdots & \cdots  & \cdots & 0  \\
   M-1 & 0 & 0 & 0 & \cdots & -M \lambda(r) & \lambda(r) \\
   M & 0 & 0 & 0 & 0 & M \lambda(r) & -M \lambda(r) \\
} 
\end{equation}
and can be written as 
\begin{equation}
\label{matrixS}
H(r) = \begin{pmatrix}
0 & {\bf 0} \\
{\bf S^0}(r) & S(r) \\
\end{pmatrix},
\end{equation}
where $S(r)$ is an $M \times M$ matrix and ${\bf S^0}(r) = -S(r) {\bf 1}_M$, where ${\bf 1}_M$ represents the $M$-dimensional column vector with every element being 1. Let 
\begin{equation}
\label{vectora}
 (a_0, {\bf a}) = (a_0, a_1, \ldots, a_M)
\end{equation}
be the $(M+1)$-dimensional row vector describing the probability of starting in one of the $M+1$ states. Since $d_v(0) = d$, we have that the $d$-th entry of ${\bf a}$ equals 1 and all the other entries equal 0. Computing the disconnection time of a node with initial degree $d$ is equivalent to computing the hitting time of the Markov chain $\mathcal{M}$ to state 0 starting from state $d$.

\begin{lemma}[{\bf Mean and law of the disconnection time}]
\label{disconnectiontime}
Consider a node $v \in V$ of initial degree $d_v(0) = d > 0$, and let the edge dynamics be such that $\mathbb{E}_{1_U}[S_e] = \mu(r)$ for each $e \in E_{MN}$. 
\begin{itemize}
\item[(i)] The disconnection time $D_v^{Q^0}$ satisfies
\begin{equation}
\mathbb{E}_{1_U} [D_v^{Q^0}] = C_d(M) \, \mu(r),
\end{equation}
with
\begin{equation}
C_d(M) = \sum_{k=1}^d \sum_{n=0}^{M-k}\frac{(M-k)!(k-1)!}{n! (M-n)!}.
\end{equation}

\item[(ii)] The law of the disconnection time $D_v^{Q^0}$ follows a phase-type distribution $\mathrm{PH}({\bf a}, S(r))$ and is given by 
\begin{equation}
\mathbb{P}_{1_U}(D_v^{Q^0} > x) = {\bf a} \, \mathrm{exp}(S(r)x)  {\bf 1}_M, \qquad x \in (0, \infty),
\end{equation}
where ${\bf a}$ and $S(r)$ are as in \eqref{vectora} and \eqref{matrixS}, respectively. In particular, the above probability equals the sum of the entries in $d$-th row of the matrix $\mathrm{exp}(S(r)x)$.
\end{itemize}
\end{lemma}

\begin{proof}
We prove the two statements separately.
\begin{itemize}
\item[(i)] Consider the Markov chain $\mathcal{M}$ described above. We know that from each state $k = 1, \dots, M$, we jump to a neighboring state with rate $M \lambda(r)$. The jump occurs exactly when the first of the $M$ possible edges changes its state. This corresponds to the minimum of $M$ i.i.d.\ exponential random variables, which is known to follow an exponential distribution with mean $\frac{\mu(r)}{M}$. 
If we denote by $x_k$ the mean hitting times of state $0$ starting from state $k$ and we set $x_0 = 0$, the vector $(x_1, \dots, x_M)$ satisfies the system of equations
\begin{equation}
\label{wn3:eqsystem}
\left\{\begin{array}{lcl}
x_1 &=& \frac{1}{M} \frac{\mu(r)}{M} + \frac{M-1}{M} \bigg(\frac{\mu(r)}{M} + x_2 \bigg) \\
x_2 &=& \frac{2}{M} \bigg( \frac{\mu(r)}{M} + x_1 \bigg) + \frac{M-2}{M} \bigg( \frac{\mu(r)}{M} + x_3 \bigg) \\
\cdots &=& \cdots \\
x_{M-1} &=& \frac{M-1}{M} \bigg( \frac{\mu(r)}{M} + x_{M-2} \bigg) + \frac{1}{M} \bigg(\frac{\mu(r)}{M} + x_M \bigg) \\
x_M &=& \frac{\mu(r)}{M} + x_{M-1},
\end{array}
\right.
\end{equation}
This allows us to compute the mean disconnection time of $v$ given its initial degree $d$. Note that for $k=1, \dots, M-1$, by writing $x_k = \frac{k}{M} x_k + \frac{M-k}{M} x_k$, we can rewrite the $k$-th equation as
\begin{equation}
\frac{k}{M} (x_k-x_{k-1}) =  \frac{M-k}{M} (x_{k+1}-x_k) + \frac{\mu(r)}{M}.
\end{equation}
If we let $X_k = x_k - x_{k-1}$ for each $k = 1, \dots, M$, each of the first $M-1$ equations can be rewritten in terms of the difference variables $X_k$ and $X_{k+1}$ as
\begin{equation}
X_k = \frac{M-k}{k} X_{k+1} + \frac{\mu(r)}{k}
\end{equation}
and the last equation becomes
\begin{equation}
X_M = \frac{\mu(r)}{M}.
\end{equation}
Starting from the initial value $X_M = \frac{\mu(r)}{M}$, we can then determine in a recursive fashion the values of $X_k$ for each $k = 1, \dots, M-1$, which are given by
\begin{equation}
X_k = \sum_{n=0}^{M-k}\frac{(M-k)!(k-1)!}{n! (M-n)!} \mu(r).
\end{equation}
By noticing that $X_1 = x_1-x_0 = x_1$, we can recover the values of the original variables $x_2 = X_2+x_1 = X_2 + X_1$, $x_3=X_3 +x_2 = X_3+X_2+X_1$, and so on. We then have that
\begin{equation}
x_d = \sum_{k=1}^d X_k = \sum_{k=1}^d \sum_{n=0}^{M-k}\frac{(M-k)!(k-1)!}{n! (M-n)!} \mu(r),
\end{equation}
which gives us the mean disconnection time of $v$ given its initial degree $d$.

\item[(ii)] The disconnection time of a node $v \in V$ of initial degree $d > 0$ is the hitting time of state 0 of the Markov chain $\mathcal{M}$ starting from state $d$. The distribution of the hitting time to the unique absorbing state, starting from any of the other finite transient states, is said to be phase-type and is denoted by $\mathrm{PH}({\bf a}, S(r))$, with ${\bf a}$ and $S(r)$ as in \eqref{vectora} and \eqref{matrixS}, respectively.

The distribution function of $D_v^{Q^0}$ is given by
\begin{equation}
\mathbb{P}_{1_U}(D_v^{Q^0} \leq x ) = \int_0^x \mathcal{P}(y) \,dy = 1 - {\bf a} \, \mathrm{exp}(S(r)x) \, {\bf 1}_M, \qquad x \in (0, \infty),
\end{equation}
where $\mathrm{exp}(\cdot)$ indicates the matrix exponential, and  
\begin{equation}
\mathcal{P}(z) = {\bf a} \, \mathrm{exp}(S(r)z) \, {\bf S^0}(r), \qquad z \in (0, \infty),
\end{equation}
with ${\bf S^0}(r)$ as in \eqref{matrixS}. Since the vector ${\bf a}$ has its $d$-th entry equal to 1 and all the other entries equal to 0, we have that the product ${\bf a} \, \mathrm{exp}(S(r)x) \, {\bf 1}_M$ equals the sum of the entries in the $d$-th row of the matrix $\mathrm{exp}(S(r)x)$. 
\end{itemize}
\end{proof}

We next show that, with high probability as $r \to \infty$, the disconnection time $D_v^{Q^0}$ takes values of the order of its mean.
\begin{lemma}[\bf Values of the disconnection time]
\label{phvalues}
Consider a node $v \in V$ of initial degree $d_v(0) = d > 0$, and let the edge dynamics be such that $\mathbb{E}_{1_U}[S_e] = \mu(r)$ for each $e \in E_{MN}$. With high probability as $r \to \infty$, the disconnection time $D_v^{Q^0}$ takes values of order $\mu(r)$, i.e., if $x=x(r) \prec \mu(r)$, then
\begin{equation}
\lim_{r \to \infty} \mathbb{P}(D_v^{Q^0} \leq x) = 0,
\end{equation}
and, if $x=x(r) \succ \mu(r)$, then
\begin{equation}
\lim_{r \to \infty} \mathbb{P}(D_v^{Q^0} \geq x) = 0.
\end{equation}
\end{lemma}

\begin{proof}
Recall that by Lemma~\ref{disconnectiontime}(ii) we have $\mathbb{P}(D_v^{Q^0} \leq x) = 1 - {\bf a} \, \mathrm{exp}(S(r)x) \, {\bf 1}_M$, where ${\bf a} \, \mathrm{exp}(S(r)x) \, {\bf 1}_M$ equals the sum of the entries in the $d$-th row of the matrix $\mathrm{exp}(S(r)x)$. Note that we can write $S(r) = \lambda(r) S' = \frac{1}{\mu(r)} S'$, where $S'$ is a tridiagonal $M \times M$ matrix. If we pick $x \prec \mu(r)$, by expanding the matrix exponential and taking the limit, we have
\begin{equation}
\lim_{r \to \infty} e^{S(r)x} = \lim_{r \to \infty} e^{\frac{x}{\mu(r)} S'} = \lim_{r \to \infty} \sum_{n =0}^{\infty} \frac{1}{n!} \bigg(\frac{x}{\mu(r)} S'\bigg)^n = I_M,
\end{equation}
where $I_M$ indicates the $M \times M$ identity matrix. Hence the sum of the entries in the $d$-th row is $1$ and we have
\begin{equation}
\lim_{r \to \infty} \mathbb{P}(D_v^{Q^0} \leq x) = \lim_{r \to \infty} 1 - {\bf a} \, \mathrm{exp}(S(r)x) \, {\bf 1}_M = 1-  1 = 0.
\end{equation}
In order to deal with the case $x \succ \mu(r)$, we need some extra definitions. An $M \times M$ matrix $A=(a_{ij})_{i,j = 1, \dots, M}$ is diagonally dominant if for each row $i = 1, \dots, M$, we have
\begin{equation}
\label{dd}
|a_{ii}| \geq \sum_{j \neq i} |a_{ij}|. 
\end{equation}
Moreover, $A$ is irreducibly diagonally dominant if it is irreducible and diagonally dominant, with strict inequality holding in \eqref{dd} for at least one row. Note that the matrix $S'$ is irredicubly diagonally dominant, since the Markov chain is irreducible, the first row satisfies the strict inequality in \eqref{dd} and all the other rows sums up to $0$. It follows that $S'$ is non-singular, hence invertible (see \cite[Theorem 1.21]{V00}). We can now use the fact that if $S'$ is invertible, then
\begin{equation}
\lim_{t \to \infty} e^{t S'}= 0_M,
\end{equation}
where $0_M$ is the $M \times M$ matrix with all entries equal to zero (see \cite[Section 2.4]{LR99}). If we pick $x \succ \mu(r)$, we can write $t = \frac{x}{\mu(r)}$ and we have that
\begin{equation}
\lim_{r \to \infty} e^{S(r)x} = \lim_{r \to \infty} e^{\frac{x}{\mu(r)} S'}= 0_M.
\end{equation}
Hence the sum of the entries in the $d$-th row is $0$ and we have
\begin{equation}
\lim_{r \to \infty} \mathbb{P}(D_v^{Q^0} \geq x) = \lim_{r \to \infty} {\bf a} \, \mathrm{exp}(S(r)x) \, {\bf 1}_M = 0.
\end{equation}
\end{proof}

\subsection{Nucleation vs.\ dynamics}
\label{ss:NvsD}

We start with the definition of forks and nucleation times, which play a crucial role in the analysis of the transition time. The term \textit{nucleation} is borrowed from statistical physics and refers to the first step in the spontaneous formation of a new structure starting from a metastable state. In our context, the new structure is represented by the state with a node in $V$ active and its neighbors in $U$ inactive.

\begin{definizione}[\bf{Forks and nucleation times}]
\label{nucleationdef}
Given a node $v \in V$, we refer to the {\it fork} of $v$ as the complete bipartite subgraph of $G$ containing only node $v$, its neighbors in $U$ and the edges between them. The time it takes the fork of $v$ to deactivate its nodes in $U$ and activate $v$ is called the {\it nucleation time} of the fork of $v$. We denote this time by $\mathcal{T}_v^Q$ when conditioning on the state of the queues ${\bf Q}$. 
\end{definizione}

In a similar way as the transition time, the nucleation times are random variables that depend on the state of the queues and on the activating forks, hence all the expectations should be interpreted as conditional expectations. 

Without loss of generality, we may consider graphs with no isolated nodes in $V$, since after time $o(1)$ we would be in such a scenario anyway. 

\begin{lemma}[{\bf Isolated nodes}]
\label{isolatednodes}
Nodes in $V$ with initial degree 0 activate in time $o(1)$ as $r \to \infty$.
\end{lemma}
\begin{proof}
Consider the situation where $\lambda(r) \prec g_V(0)$, i.e., the dynamics are slower than the average time it takes for the activation clock of nodes in $V$ to tick. Then a node $v \in V$ with initial degree 0 activates as soon as its clock ticks, hence in time $o(1)$. 
Next, consider the situation where the dynamics are very fast, $\lambda(r) \succ g_V(0)$. Then a node $v \in V$ with initial degree 0 might be blocked by some active neighbors in $U$ by the time its activation clock ticks for the first time. Recall that $|U| =M$ and note that there are $2^M$ possible configurations of edges connecting $v$ to $U$. Each time the activation clock of $v$ ticks, the probability of being in each of the possible configurations tends to the uniform probability $1/2^M$ as $r \to \infty$. Therefore, after a finite number of attempts, $v$ eventually activates. Since each tick of the activation clock of $v$ takes time $o(1)$, $v$ activates in time $o(1)$.
Lastly, consider the situation where $\lambda(r) \asymp g_V(0)$. If the activation clock of a node $v \in V$ with initial degree 0 ticks before any of its potential edges appear, then $v$ activates in time $o(1)$. Otherwise, each subsequent activation attempt will not be successful unless the edge configuration is such that $v$ has no neighbors. In other words, $v$ can activate only when the Markov chain describing how its degree changes over time is in state 0. In this case, $v$ activates with a probability that at time $t$ is given by $\frac{g_V(t)}{g_V(t) + M \lambda(r)} > 0$ as $r \to \infty$. Since $\lambda(r)^{-1} = o(1)$, by using similar arguments as in the proof of Lemma~\ref{disconnectiontime}, the time it takes for the Markov chain to return to state 0 when starting from state 0 is $o(1)$. Hence, $v$ has the chance to activate with positive probability every period of time $o(1)$. Therefore, after a finite number of attempts, $v$ eventually activates in time $o(1)$.
\end{proof}

We call \textit{activation time} of $v \in V$ the time it takes for $v$ to activate. Depending on the dynamics, this can be given either by its nucleation time $\mathcal{T}_v^{Q^0}$ or by its disconnection time $D_v^{Q^0}$. When the dynamics are fast enough, nodes in $V$ eventually activate because their clocks tick and no edges connect them to nodes in $U$. On the other hand, when the dynamics are particularly slow, it is more likely for nodes in $V$ to activate through the nucleation of its fork, and the network tends to behave as if the edges were frozen at the initial configuration. In between these two scenarios the dynamics are more interesting and, depending on the speed, we distinguish between different behaviors. Proposition~\ref{NvsD} below describes the competition between the nucleation and the dynamics.

\begin{proposizione}[{\bf Competition: nucleation vs.\ dynamics}]
\label{NvsD}
Let $v \in V$ be a node of minimum degree at time $t=0$, with $d_v(0) = d > 0$.
\begin{itemize}
\item[(i)] If $\lambda(r) \succ r^{-(1 \wedge \beta(d-1))}$, then, with high probability as $r \to \infty$, the activation time of $v$ is given by its disconnection time, i.e., 
\begin{equation}
\lim_{r \to \infty} \mathbb{P}_{1_U}(D_v^{Q^0} < \mathcal{T}_v^{Q^0}) = 1.
\end{equation}
\item[(ii)] If $\lambda(r) \asymp r^{-(1 \wedge \beta(d-1))}$, then the activation time of $v$ is given either by its nucleation time with positive probability or by its disconnection time with positive probability. 
\item[(iii)] If $\lambda(r) \prec r^{-(1 \wedge \beta(d-1))}$, then, with high probability as $r \to \infty$, the activation time of $v$ is given by its nucleation time, i.e., 
\begin{equation}
\lim_{r \to \infty} \mathbb{P}_{1_U}( \mathcal{T}_v^{Q^0} < D_v^{Q^0}) = 1.
\end{equation}
\end{itemize}
\end{proposizione}

\begin{proof}
Recall that $\mu(r) = \lambda(r)^{-1}$. By Lemma~\ref{phvalues}, with high probability as $r\ \to \infty$, $D_v^{Q^0}$ takes values of order $\mu(r)$. Recall also that, depending on the relation between $\beta$ and $d$, the nucleation time $\mathcal{T}_v^{Q^0}$ is given by an exponential random variable with mean of order $r^{\beta(d-1)}$, by a polynomial random variable with mean of order $r$, or by $T_U(r)$, which is the average time it takes for the queue lengths at nodes in $U$ to hit zero. Hence, with high probability as $r \to \infty$, $\mathcal{T}_v^{Q^0}$ takes values of order $r^{1 \wedge \beta(d-1)}$. It is therefore immediate to distinguish between the three cases.
\begin{itemize}
\item[(i)] Since $\mu(r) \prec r^{1 \wedge \beta(d-1)}$, with high probability as $r \to \infty$, $v$ activates due to absence of edges.
\item[(ii)] Since $\mu(r) \asymp r^{1 \wedge\beta(d-1)}$, there is a competition between the nucleation time $\mathcal{T}_v^{Q^0}$ and the phase-type random variable $D_v^{Q^0}$. Depending on their parameters, each of them can occur before the other with positive probability.
\item[(iii)] Since $\mu(r) \succ r^{1 \wedge\beta(d-1)}$, with high probability as $r \to \infty$, $v$ activates through the nucleation of its fork.
\end{itemize}
\end{proof}

\section{Proofs of the main results}
\label{proofs}
We are now able to prove the main results of the paper. We first prove Theorem~\ref{meantrtime} for dynamic bipartite graphs with queue-based activation rates. Then we adapt the arguments in \cite{BdHNS18} and \cite{BdHNS20} to prove Theorems~\ref{adaptedthmpaper1}--\ref{adaptedmeantrtime} for the special case of fixed activation rates. 

\subsection{Proof for the model with queue-based activation rates}
In this section we prove Theorem~\ref{meantrtime}.

\begin{proof}[Proof of Theorem~\ref{meantrtime}] 
We analyze the different types of dynamics separately.

\begin{itemize}
\item[(FD)] Assume $\lambda(r) \to \infty$ as $r \to \infty$. By Lemma~\ref{disconnectiontime}, the mean disconnection time of a node in $V$ is of order $\lambda(r)^{-1} = o(1)$. Let $v_1$ be the first node activating in $V$ and consider the event $\mathcal{B}_{v_1} = \{ D_{v_1}^{Q^0} < \mathcal{T}_{v_1}^{Q^0} \}$. The mean activation time of $v_1$ is then given by $\mathbb{E}_{1_U}[D_{v_1}^{Q^0} |\mathcal{B}_{v_1} ] \mathbb{P}_{1_U}(\mathcal{B}_{v_1}) + \mathbb{E}_{1_U}[\mathcal{T}_{v_1}^{Q^0}|\mathcal{B}_{v_1}^C] \mathbb{P}_{1_U}(\mathcal{B}_{v_1}^C)$. Even though we know from Proposition~\ref{NvsD}(i) that $\lim_{r \to \infty}\mathbb{P}_{1_U}(\mathcal{B}_{v_1}^C) = 0$, a priori this term may still affect the mean activation time, since the conditional expectation may be substantial. Since we are interested in the typical behavior of the network, we focus only on the first term of the sum. We deal with the remaining activating nodes in $V$ in a similar way: when computing the mean activation time of the $k$-th activating node $v_k$, we condition on the event $\mathcal{B}_{v_k} = \{  D_{v_k}^{Q^{k-1}} < \mathcal{T}_{v_k}^{Q^{k-1}}\}$, where we consider the updated queue lengths ${\bf Q^{k-1}}$ after $k-1$ nodes in $V$ have been activated; generalizing Proposition~\ref{NvsD}(i) to node $v_k$, we have that $\lim_{r \to \infty}\mathbb{P}_{1_U}(\mathcal{B}_{v_k}^C) = 0$. Hence, with high probability as $r \to \infty$, the activation time of each node in $V$ is given by its disconnection time, which has mean of order $\lambda(r)^{-1} = o(1)$. Denote by $\mathcal{B} = \bigcap_{i=1}^{|V|} \mathcal{B}_{v_i}$ the event that captures the typical behavior of the competition and recall that $\mathcal{F}$ is the event that the network follows the algorithm (see Remark~\ref{typical}(i)). Denote by $\mathcal{E} = \mathcal{F} \cap \mathcal{B}$ the event that captures the typical behavior of the network and note that $\mathcal{E}$ occurs with high probability as $r \to \infty$. Conditional on this event $\mathcal{E}$, the transition time of $G(\cdot)$ with initial queue lengths ${\bf Q^0}$ satisfies
\begin{equation}
\mathbb{E}_{1_U}[\mathcal{T}_{G(\cdot)}^{Q^0}\,  | \, \mathcal{E} ] \asymp \lambda(r)^{-1} = o(1), \qquad r \to \infty.
\end{equation}

\item[(RD)] Assume $\lambda(r) = C \in (0,\infty)$. By Lemma~\ref{disconnectiontime}, the mean disconnection time of a node in $V$ is of order $\lambda(r)^{-1} = \mathcal{O}(1)$. Note that nodes in $V$ of degree $1$ activate  in $\mathcal{O}(1)$ either because their only neighbor deactivates or due to absence of edges. Moreover, by Proposition~\ref{NvsD}(i), with high probability as $r \to \infty$, nodes in $V$ of degree greater than $1$ activate due to absence of edges in a time with mean of order $\lambda(r)^{-1}=\mathcal{O}(1)$. The argument is analogous to the one used above in the FD proof, where we condition on an event $\mathcal{B}$ that captures the typical behavior of the competition. Conditional on the event $\mathcal{E} = \mathcal{F} \cap \mathcal{B}$, which captures the typical behavior of the network and occurs with high probability as $r \to \infty$, the transition time of $G(\cdot)$ with initial queue lengths ${\bf Q^0}$ satisfies
\begin{equation}
\mathbb{E}_{1_U}[\mathcal{T}_{G(\cdot)}^{Q^0} \, | \, \mathcal{E}] \asymp \lambda(r)^{-1} = \mathcal{O}(1), \qquad r \to \infty.
\end{equation}

\item[(SDnc)] Assume $\lambda(r) \to 0$ as $r \to \infty$ with $\lambda(r) \prec r^{-(1 \wedge \beta(d^*-1))}$. In order to capture the typical behavior of the competition, when computing the mean activation time of the $k$-th activating node $v_k \in V$, we now condition on a new event $\mathcal{B}_{v_k} = \{  \mathcal{T}_{v_k}^{Q^{k-1}} < D_{v_k}^{Q^{k-1}} \}$. Indeed, differently from the FD and RD proofs, Lemma~\ref{disconnectiontime} shows that the mean disconnection time of $v_k$ is of order larger than $r^{1 \wedge \beta(d^*-1)}$, and Proposition~\ref{NvsD}(iii) implies that $\lim_{r \to \infty} \mathbb{P}_{1_U} (\mathcal{B}^C_{v_k}) = 0$. Hence, with high probability as $r \to \infty$, each node in $V$ activates through the nucleation of its fork, which is at most of order $r^{1 \wedge \beta(d^*-1)}$. The dynamics are very slow, almost frozen, and so they do not affect the nucleation of the forks and the transition. Write $\mathcal{B} = \bigcap_{i=1}^{|V|} \mathcal{B}_{v_i}$. Conditional on the event $\mathcal{E} = \mathcal{F} \cap \mathcal{B}$, the transition time of $G(\cdot)$ with initial queue lengths ${\bf Q^0}$ satisfies Theorem 3.3 in \cite{BdHNS20} and the network behaves as if there were no dynamics.

\item[(SDc)] Assume $\lambda(r) \to 0$ as $r \to \infty$ with $\lambda(r) = r^{-\alpha}$, with $0 < \alpha \leq 1 \wedge \beta(d^*-1)$. This is the most interesting type of dynamics, since it competes with the fork nucleations. The activation of the nodes in $V$ can occur both because of the absence of their edges and because of the nucleation of their forks. Denote by $\hat{d}$ the largest integer such that $\beta(\hat{d} -1) < \alpha$. Recall that $\mathcal{A}$ denotes the set of admissible paths and fix a path $a \in \mathcal{A}$. Consider the sequence of activating nodes along the path $a$ up to the step in which the degree is larger than $\hat{d}$. Say that at step $k$ we have $\bar{d}_k > \hat{d}$. Consider only the first $k-1$ steps. We indicate by $\mathcal{F}_a(\alpha)$ the event that the network follows the path $a \in \mathcal{A}$ until time scale $r^{\alpha}$. On time scale $r^{\alpha}$ the dynamics start competing with the nucleations, and the order of activation of the remaining nodes described by the algorithm is not preserved anymore. In other words, the order of activation of nodes in $V$ follows the order of activation of the path $a$ only for the first $k-1$ nodes. Consider at each step a node $v_j$ with minimum degree $\bar{d}_j$ for $j=1, \ldots, k-1$. When computing the mean activation time of $v_j$, we condition on an event $\mathcal{B}_{v_j}  = \{ \mathcal{T}_{v_j}^{Q^{j-1}} < D_{v_j}^{Q^{j-1}} \}$. Since by Lemma~\ref{disconnectiontime} the mean disconnection time of each $v_j$ is of order $r^{\alpha}$, by Proposition~\ref{NvsD}(iii) we have $\lim_{r \to \infty} \mathbb{P}_{1_U} (\mathcal{B}^C_{v_j})  = 0$. Hence, with high probability as $r \to \infty$, the activation time of each of the first $k-1$ activating nodes in $V$ is given by their nucleation time, which has mean of order less than or equal to $r^{1 \wedge\beta(\hat{d}-1)}$. We denote by $\mathcal{B}_a(\alpha) = \bigcap_{i=1}^{k-1} \mathcal{B}_{v_j}$ the event that captures the typical behavior of the competition for the first $k-1$ nodes. We treat the subcritical, critical and supercritical regimes separately.

\begin{itemize}
\item[(I)] $\beta \in (0, \frac{1}{d^*-1})$: subcritical regime. 
We have $0 < \alpha \leq \beta(d^*-1) < 1$. The activation time of the next activating node is of order $r^{\alpha}$. It cannot be of smaller order since at step $k$ we have $\bar{d}_k > \hat{d}$ by construction. It cannot be of higher order either since the disconnection time of any of the remaining nodes is of order $r^{\alpha}$. 
After this activation, there might be nodes whose degree has decreased and whose nucleation time is of a smaller order. When we sum the mean activation times of the nodes in $V$ to compute the mean transition time, we see that these nodes will not contribute significantly as $r \to \infty$. All the remaining nodes are likely to activate in any possible order, but none of them will have an activation time of order larger than $r^{\alpha}$.
To conclude, the order of activation of nodes in $V$ follows the path $a$ as long as the nucleation times associated to the nodes are of order smaller than $r^{\alpha}$. After that, the remaining nodes can activate with positive probability in any order with an activation time of order at most $r^{\alpha}$. Hence, conditional on the event $\mathcal{E}_a(\alpha) = \mathcal{F}_a(\alpha) \cap \mathcal{B}_a(\alpha)$, the transition time satisfies 
\begin{equation}
\mathbb{E}_{1_U}[\mathcal{T}_{G(\cdot)}^{Q^0} \,|\, \mathcal{E}_a(\alpha)] \asymp r^{\alpha}, \qquad r \to \infty.
\end{equation}
Next, denote by $\mathcal{F}(\alpha)$ the event that the network follows the algorithm until time scale $r^{\alpha}$ and by $\mathcal{B}(\alpha)$ the event that captures the typical behavior of the competition until time scale $r^{\alpha}$. Then, conditional on the event $\mathcal{E}(\alpha) = \mathcal{F}(\alpha) \cap \mathcal{B}(\alpha)$, the transition time satisfies
\begin{equation}
\mathbb{E}_{1_U}[\mathcal{T}_{G(\cdot)}^{Q^0} \, | \, \mathcal{E}(\alpha)] \asymp r^{\alpha}, \qquad r \to \infty.
\end{equation}

\item[(II)] $\beta = \frac{1}{d^*-1}$: critical regime.
For $0 < \alpha < 1$, the situation is the same as in the subcritical regime described above. 
For $\alpha =1$, the activation time of the next activating node is of order $r$. The order of activation of nodes in $V$ follows the algorithm as long as the nucleation times associated to the nodes are of order smaller than $r$. After that, the remaining nodes can activate with positive probability in any order with an activation time of order at most $r$. Denote by $\mathcal{F}(1)$ the event that the network follows the algorithm until time scale $r$ and by $\mathcal{B}(1)$ the event that captures the typical behavior of the competition until time scale $r$. Then, conditional on the event $\mathcal{E}(1) = \mathcal{F}(1) \cap \mathcal{B}(1)$, the transition time satisfies
\begin{equation}
\mathbb{E}_{1_U}[\mathcal{T}_{G(\cdot)}^{Q^0} \, | \,\mathcal{E}(1)  ] \asymp r, \qquad r \to \infty.
\end{equation}
Moreover, we know from \cite{BdHNS18} that with high probability as $r \to \infty$ the nucleation time of a subcritical or critical fork is smaller than $T_U(r)$. This implies that even if there are nodes with disconnection time of order $r$ but larger than $T_U(r)$, with high probability as $r \to \infty$ they activate through the nucleation of their fork. Hence,
\begin{equation}
\lim_{r \to \infty} \mathbb{P}_{1_U} \big(\mathcal{T}_{G(\cdot)}^{Q^0} < T_U(r) \, | \, \mathcal{E}(1) \big)=1.
\end{equation}

\item[(III)] $\beta \in (\frac{1}{d^*-1}, \infty)$: supercritical regime.
For $0 < \alpha < 1$, the situation is the same as in the subcritical regime described above.
For $\alpha = 1$, with positive probability there is at least one supercritical node that activates in time $T_U(r)$. Indeed, depending on whether its disconnection time (which is of order $r$) is larger or smaller than $T_U(r)$, then the node activates in $T_U(r)$ because the queue lengths at nodes in $U$ hit zero  or in a time smaller than $T_U(r)$ because of absence of edges, respectively. Both scenarios can occur with positive probability. Hence, conditional on the event $\mathcal{B}(1)$, the transition time satisfies 
\begin{equation}
\lim_{r \to \infty} \mathbb{P}_{1_U} \big(\mathcal{T}_{G(\cdot)}^{Q^0} = T_U(r) \,|\, \mathcal{B}(1)\big) = 1- \lim_{r \to \infty} \mathbb{P}_{1_U} \big(\mathcal{T}_{G(\cdot)}^{Q^0} < T_U(r) \,|\, \mathcal{B}(1)\big) > 0.
\end{equation}
Recall from \cite{BdHNS20} that in the supercritical regime the behavior of the transition time does not depend on which path the network follows, hence it is not needed to condition also on the event $\mathcal{F}(1)$.
\end{itemize}
\end{itemize}
\end{proof}

\subsection{Proofs for the model with fixed activation rates}

In this Section we prove Theorems~\ref{adaptedthmpaper1}--\ref{adaptedmeantrtime}.

\begin{proof}[Proof of Theorem~\ref{adaptedthmpaper1} (Complete bipartite graphs)] 
It follows from \cite[Sections 4.1-4.2]{BdHNS18}. We compute the critical time scale and the mean transition time using fixed activation rates instead of time depending ones. In both the critical and subcritical regimes, the pre-factor turns out to be $\frac{1}{|U|}$ and the law is exponential. In the critical regime, we know that the queue lengths decrease significantly after a time of order $r$. However, this does not affect the transition time, since now the activation rates do not depend on the queue lengths. In the supercritical regime, we still have the same behavior as in the model with queue-based activation rates. Indeed, when the queue lengths at nodes in $U$ hit zero, the nodes in $U$ deactivate by assumption and the transition occurs.
\end{proof}

\begin{proof}[Proof of Theorem~\ref{adaptedthmpaper2} (Arbitrary bipartite graphs)]
The claims follow from Theorem~\ref{adaptedthmpaper1} and from the analysis of the algorithm and the next nucleation times in \cite[Sections 2 and 4.2]{BdHNS20}. We derive the mean transition time along the paths generated by the algorithm by computing the next nucleation times at each step. In the subcritical regime, the nucleation times of nodes in $V$ are all exponentially distributed and independent of each other. Indeed, the activation rates are the same, independently of the queue lengths decreasing over time. At each step $k$, the next nucleation time is the minimum of $n_k$ i.i.d.\ exponential random variables, and hence its mean exhibits the term $f_k=\frac{1}{n_k}$ in the pre-factor. In the critical regime, the pre-factor of the mean transition time along each path must be below the value $\frac{\gamma_U}{c-\rho_U}$, otherwise the supercritical regime applies and the transition occurs because the queue lengths at nodes in $U$ hit 0. If we assume that $\frac{\gamma_U}{c-\rho_U} > 1$, then the nucleation of a fork occurs before the queue lengths at nodes in $U$ hit zero. We are able to derive the law of the transition time along each path for both the subcritical and critical regimes. Both are described by convolutions of the exponential laws of the next nucleation times of the activating nodes in $V$. In the supercritical regime, we have the same behavior as in the model with queue-based activation rates.
\end{proof}

\begin{proof}[Proof of Theorem~\ref{adaptedmeantrtime} (Dynamic bipartite graphs)]
The claim follows from Theorem~\ref{adaptedthmpaper2} and the intuition behind Proposition~\ref{NvsD}. The order of the mean transition time in the model with fixed activation rates is the same as in the model with queue-based activation rates. The dynamics compete with the nucleations of the nodes in the same way, depending on the speed. The different types of dynamics, fast, regular and slow, lead to the same results as in Theorem~\ref{meantrtime}.
\end{proof}

\section{Conclusion}

In this conclusive section we summarize the results obtained in this paper and discuss the main challenges.  

We explore a dynamic graph model in order to capture some effects of user mobility in wireless random-access networks. Nodes can be active or inactive and represent transmitter-receiver pairs, while edges model the interference between nearby transmissions in such a way that two neighboring nodes are not allowed to be active simultaneously. Following the work in \cite{BdHNS18} and \cite{BdHNS20}, we focus on bipartite interference graphs and we introduce edge dynamics by allowing the edges to appear and disappear over time according to Poisson i.i.d.\ clocks attached to each of them. Motivated by the fact that metastability properties are crucial in analyzing starvation behavior and designing mechanisms to improve the performance of the network, we study the transition time between the two dominant states where one part of the network is active, $U$, and the other part is inactive, $V$.

We first consider a model with queue-based activation rates, which turns out to be quite challenging since it leads to two levels of complexity, driven by the queue dependences of the activation rates and by the edge dynamics. We then also consider a model with fixed activation rates, since not much is known about the effects of the dynamics even for this simplified model.

The parameter $r \to \infty$ represents the heavy load scenario where the queue lengths (hence the activation rates) grow large. We are able to determine how the dynamics affect the mean transition time by distinguishing between different speeds, namely, fast, regular and slow dynamics. The most interesting scenarios arise in the case of slow dynamics, in particular when the rate of the dynamics is competitive, in the sense that some nodes in $V$ can activate both because their neighbors are simultaneously inactive and because the edges connecting them to their neighbors disappear. Our main result focuses on the typical behavior of the network and states that, if the rate of the dynamics is $\lambda(r)$, then the mean transition time is of order $\lambda(r)^{-1}$. For static bipartite graphs, a randomized algorithm generates all the possible orders of activation of the nodes in $V$ and allows us to compute precise asymptotics for the mean transition time. For dynamic bipartite graphs instead, the situation is different. Each node in $V$ whose nucleation time is of smaller order than $\lambda(r)^{-1}$ activates through the nucleation of its fork. On time scale $\lambda(r)^{-1}$ the dynamics start competing with the nucleations and the order of activation of the remaining nodes described by the algorithm is not preserved anymore. 

The evolution of the network is captured by a continuous-time Markov process, which we refer to as the graph evolution process. It keeps track of the state of the nodes, the queue lengths and the edges between nodes, hence it is crucial to describe how the (active) degrees of the nodes change over time. The main challenge in describing the graph evolution process is that on time scale $\lambda(r)^{-1}$ any of the remaining nodes could activate next with positive probability. The activation of a node due to absence of edges is captured by the scenario in which its active degree hits $0$, while the activation of a node through the nucleation of its fork depends on the aggressiveness of the activation rates and on the number of active neighbors. Both types of activation are determined by the degrees of the nodes, hence the transition time is strongly related to the graph evolution process. The complicated nature of the process prevents us from deriving an explicit formula for the pre-factor of the mean transition time. Indeed, in order to give precise asymptotics, a better understanding and control of how the degrees of the nodes change over time is required. This goes beyond the scope of the present paper, but it could lead to interesting future research directions.

\vspace{1cm}
\noindent 
\emph{Acknowledgments}: The research in this paper was supported by the Netherlands Organisation for Scientific Research (NWO) through Gravitation-grant NETWORKS-024.002.003. The author thanks professors S.C.\ Borst (Eindhoven University of Technology), F. den Hollander (Leiden University) and F.R.\ Nardi (University of Florence) for the helpful suggestions and the useful discussions.



\appendix
\renewcommand*{\thesection}{\Alph{section}}


\end{document}